\documentclass[12pt]{article}
 \usepackage{amsthm,amsmath,amssymb}
\newcommand{\gbar}{\overline{G}}
\newcommand{\dsp}{\displaystyle}
\theoremstyle{plain}
\newtheorem{Thm}{Theorem}
\newtheorem{Cor}[Thm]{Corollary}

\newtheorem{lemma}[Thm]{Lemma}
\newtheorem{Prop}[Thm]{Proposition}
\newtheorem{Def}[Thm]{Definition}

\newtheorem{remark}[Thm]{Remark}
\newtheorem{example}[Thm]{Example}

\newcommand{\ngone}{\text{NG-1 }}
\newcommand{\ngtwo}{\text{NG-2 }}
\newcommand{\ngthree}{\text{NG-3 }}

\newcommand{\NGcal}{\cal{N}\cal{G}}
\newcommand{\NGone}{$\cal{N}\cal{G}$-1}
\newcommand{\NGtwo}{$\cal{N}\cal{G}$-2}
\newcommand{\NGthree}{$\cal{N}\cal{G}$-3}

\newcommand{\NGonecaptwo}{(\NGone \ $\cap$ \NGtwo)}
\newcommand{\NGonecuptwo}{(\NGone \ $\cup$ \NGtwo)}
\newcommand{\NGonemtwo}{(\NGone \ $-$ \NGtwo)}
\newcommand{\NGtwomone}{(\NGtwo \ $-$ \NGone)}

\date{June 1, 2015 }
\title{Split Graphs and Nordhaus-Gaddum Graphs}

\author{Christine Cheng\\
\small Dept. of Computer Science\\
\small Univ. of Wisconsin at Milwaukee\\
\small Milwaukee, WI 53201\\
\small \tt ccheng@uwm.edu\
\and
  Karen L. Collins\\
\small Dept. of Mathematics and Computer Science\\
\small Wesleyan University\\
\small Middletown CT 06459-0128\\
\small\tt kcollins@wesleyan.edu\
\and
Ann N. Trenk\\
\small Department of Mathematics\\
\small Wellesley College\\
\small Wellesley MA 02481\\
\small\tt atrenk@wellesley.edu
}

\begin{document}
\maketitle

\begin{abstract} 

A graph $G$ is an NG-graph if $\chi(G) + \chi(\gbar) = |V(G)| + 1$.  We characterize NG-graphs solely from degree sequences leading to a linear-time recognition algorithm.  We also explore the connections between NG-graphs and split graphs.  	There are three types of NG-graphs and  split graphs can also be divided naturally into two categories, balanced and unbalanced. We characterize each  of these five classes by degree sequence.      We construct bijections between classes of NG-graphs and balanced and unbalanced split graphs which, together with the known formula for the number of split graphs on $n$ vertices, allows us to compute the sizes of each of these classes.  Finally, we provide a bijection between unbalanced split graphs on $n$ vertices and split graphs on $n-1$ or fewer vertices providing evidence for our conjecture that the rapid growth in the number of split graphs comes from the balanced split graphs. 

\end{abstract}

\bibliographystyle{plain} 


%\subsection{Preliminaries}
%\subsection{NG-graphs and their characterizations}
%\bibliographystyle{plain} 

\bigskip
\noindent

\section{Introduction}

 For a graph $G$, 
  the number of vertices in a largest clique in $G$ is denoted by $\omega(G)$  and    the number of vertices in a largest  stable set (independent set) in $G$ is denoted by $\alpha(G)$.  We denote the complement of $G$   by $\gbar$ and the graph induced in $G$ by $X \in V(G)$  by $G[X]$.   We write $nbh(x)$ to denote the set of vertices adjacent to vertex $x$.
  For a set of graphs, $\cal C$, we denote by ${\cal C}_n$ the set of  graphs in $\cal C$ with $n$ vertices. 

%The paper is organized as follows.  In the remainder of this section, we provide background on NG-graphs and the ABC-partition of a graph, and give a new characterization of NG-graphs.  In Section 2, we study split graphs and their connection to NG-graphs. We define what it means for a split graph to be balanced or unbalanced and use the perspective of NG-graphs to prove results about split graphs.  We characterize the different types of NG-graphs by their degree sequences in Section 3.  This leads to characterizations of balanced and unbalanced split graphs be their degree sequences and linear time recognition algorithms for all five classes.  Finally, in Sections 4 and 5 we construct bijections between classes of NG-graphs and balanced and unbalanced split graphs, enabling us to count the sizes of these classes. 

 %\subsection{NG-graphs and their characterizations}
 
A well-known  theorem by Nordhaus and Gaddum \cite{NoGa56} states that  the following is true for any graph $G$: $$ 2 \sqrt{|V(G)|}   \leq \chi(G) + \chi(\bar{G}) \leq  |V(G)| + 1.$$   We call $G$ a {\it Nordhaus-Gaddum graph} or {\it NG-graph}  if $G$ satisfies the maximum value of this inequality; i.e.,   $ \chi(G) + \chi(\bar{G})  =   |V(G)| + 1.$  Finck \cite{Fi66}  and Starr and Turner \cite{StTu08} provide  two different  characterizations of NG-graphs.  More recently,  Collins and Trenk  \cite{CoTr13}  define  the ABC-partition of a graph and characterize  NG-graphs in terms of this partition.

\begin{Def}\rm 
For a graph $G$, the  \emph{ABC-partition} of $V(G)$  (or of $G$) is  

 $A_G = \{v \in V(G): deg(v) = \chi(G) -1\}$
 
 $B_G = \{v \in V(G): deg(v) > \chi(G) -1\}$
 
 $C_G = \{v \in V(G): deg(v) < \chi(G) -1\}$.   
 
 When it is unambiguous, we write $A=A_G$, $B=B_G$, $C=C_G$.
  
\label{ABC-def}
\end{Def}

% For $V' \subseteq V(G)$, let $G[V']$ denote the subgraph induced by $V'$ in $G$.  

\begin{Thm} \label{NG-charac}  {\rm (Collins and Trenk \cite{CoTr13})}
 A graph $G$  is an NG-graph if and only if its ABC-partition satisfies
  
 (i) $A \neq \emptyset$ and  $G[A]$ is a clique, a stable set, or a 5-cycle
 
 (ii) $G[B]$ is a clique
 
 (iii)  $G[C]$ is a stable set
 
 (iv)  $uv \in E(G)$ for all $u \in A$, $v \in B$ 
 
 (v)  $uw \not\in E(G)$ for all $u \in A$, $w \in C$.
 
\label{NG-equal-thm}
\end{Thm}

By (i) of Theorem~\ref{NG-charac}, there are  three possible forms of an NG-graph.  (See Figure~\ref{NG-fig}.) 

\begin{Def} \rm We say that $G$  is an \ngone  graph if $G[A]$ is a clique, an \ngtwo graph if  $G[A]$ is a stable set, and an \ngthree graph if  $G[A]$ is a 5-cycle.   We also let \NGone  \ be the set of \ngone graphs and likewise define the sets \NGtwo \  and \NGthree.
\label{NG-123-def}
\end{Def}

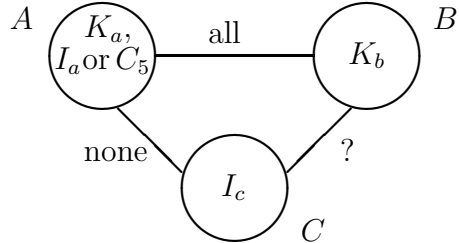
\begin{figure}[thbp]
\begin{center}

\begin{picture}(200,100)(0,30)
\thicklines
\put(50,100){\circle{40}} 
\put(100,50){\circle{40}} 
\put(150,100){\circle{40}} 

 \put(15,110){$A$}
 \put(175,110){$B$}
 \put(125,30){$C$}
 
 \put(95,47){$I_c$}
  \put(143,97){$K_b$}
\put(43,107){$K_a$,}
 \put(32,95){$I_a$\hspace{-.05in} or }

 \put(55,95){$C_5$}
 
 \put(70,100){\line(1,0){60}}
 \put(90,105){all}

\put(56,80){\line(1,-1){24}}
\put(43,60){none}

\put(144,80){\line(-1,-1){24}}
\put(140,60){?}
  
\end{picture}

\end{center}
\caption{The forms of an NG-graph}
 
\label{NG-fig}
\end{figure}

%By (i) of Theorem~\ref{NG-charac}, there are  three possible forms of an NG-graph.  (See Figure~\ref{NG-fig}.) 

%\begin{Def} \rm We say that $G$  is an \ngone  graph if $G[A]$ is a clique, an \ngtwo graph if  $G[A]$ is a stable set, and an \ngthree graph if  $G[A]$ is a 5-cycle.   We also let \NGone  \ be the set of \ngone graphs and likewise define the sets \NGtwo \  and \NGthree.
%\label{NG-123-def}
%\end{Def}

%By (i) of Theorem~\ref{NG-charac}, there are  three possible forms of an NG-graph.  (See Figure~\ref{NG-fig}.) 
%and defined as follows.
%\begin{Def} \rm
% We say that $G$  is an {\it \ngone  graph} if $G[A]$ is a clique, an {\it \ngtwo graph} if  $G[A]$ is a stable set, and an {\it \ngthree graph} if  $G[A]$ is a 5-cycle.   We also let \NGone  \ be the set of \ngone graphs and likewise define the sets \NGtwo \  and \NGthree.
%\label{NG-123-def}
%\end{Def}

The characterization in Theorem~\ref{NG-charac}  not only provides a clear description of NG-graphs but it also   lends itself to an $O(|V(G)|^3)$-time recognition algorithm for NG-graphs  \cite{CoTr13}. 
% even though the ABC-partition makes use of $\chi(G)$.  
More importantly for this article, it shows that NG-graphs are related to split graphs and pseudo-split graphs. 
\medskip

A {\it split graph} is a graph $G$ whose vertex set can be partitioned as $V(G) = K \cup S$, where $K$ induces    a clique and $S$ induces a stable set in $G$.  A detailed introduction to this class appears in \cite{Go80}.   Split graphs are  a well-known class of perfect graphs, and thus $\chi(G) = \omega(G)$ for split graphs.   Split graphs  also have elegant characterization theorems.  F\"{o}ldes and Hammer \cite{FoHa77} give a forbidden subgraph characterization of split graphs as those graphs with no induced $2K_2$, $C_4$ or $C_5$.   
%show that    $G$ is a split graph if and only if both $G$ and $\overline{G}$ are chordal (i.e., have no induced $C_n$ for $n \ge 4$).  In addition, they 
%Equivalently,  F\"{o}ldes and Hammer \cite{FoHa77} showed that  a graph is  split if and only if it does not contain $C_4$, $2 K_2$ or $C_5$ as  induced subgraphs.    
 Split graphs also have a degree sequence characterization  due to Hammer and Simeone \cite{HaSi81} which we present in Theorem~\ref{ham-sim-thm}.   This latter characterization implies that split graphs can be recognized in linear time.   

Bl\'{a}zsik et al. \cite{BlHuPlTu93}     consider  the class of graphs that do not contain $C_4$ and $2 K_2$ as induced subgraphs,  later referred to as {\it pseudo-split graphs} \cite{MaPr94}.  
%defined a  {\it pseudo-split graph} as a graph that does not contain $C_4$ and $2 K_2$ as induced subgraphs.   %Like the class of split graphs, pseudo-split graphs are self-complementary. 
They  show  that like split graphs,  pseudo-split graphs can   be defined in terms of   vertex sets  partitions.  In particular,  a graph $G$ is a pseudo-split graph if and only if $V(G)$ can be partitioned into three parts so that (i) first part is either empty or induces a 5-cycle, the second part  a clique,  the third part a stable set and (ii) whenever the first part is a 5-cycle,  every vertex in the first part is adjacent to every vertex in the second part but there are no edges between the first part and the third part.  Interestingly,  Bl\'{a}zsik et al.  also note  that pseudo-split graphs are  almost  extremal in terms of the Nordhaus-Gaddum inequality because for any such graph $G$,  $\chi(G) + \chi(\bar{G}) \ge |V(G)|$.   In the process of proving this result,  they show  that  if  $G$ contains an induced 5-cycle then $\chi(G) + \chi(\bar{G}) \ge |V(G)| + 1$ and   thus  $G$ is an NG-graph.   

The next result follows from Theorem~\ref{NG-charac}  and the characterization of pseudo-split graphs discussed above.  In particular, a graph   is an   NG-3 graph if and only if it is a pseudo-split graph  containing  an induced 5-cycle. 
%The following proposition now follows directly from the definitions and forbidden graph characterizations.
 \begin{remark} 
A graph is a pseudo-split graph if and only if it is a split graph or an  NG-3 graph. 
\label{pseudo-rem}
\end{remark}

 \begin{Prop} Let $G$ be an NG-graph.  Then $G$ is a split graph if and only if 
  $G \in $ \NGonecuptwo.
 
\label{type3-rem}
\end{Prop}

\begin{proof}
By definition,  \ngthree graphs contain an induced $C_5$, hence are not  split graphs.   Now suppose $G \in $ \NGonecuptwo \ and   let its $ABC$-partition be $V(G) = A \cup B \cup C$.  If $G\in$ \NGone,  let   $K = A \cup B$ and $S = C$  and otherwise, $G \in$ \NGtwo, and we  let    $K =  B$ and $S = A \cup C$.   In either case, using Theorem~\ref{NG-charac}, we get a  partition of $V(G)$ into a clique $K$ and a stable set $S$, so $G$ is a split graph.
\end{proof}
 
 Not all split graphs are NG-graphs.  Indeed, the relationship between these classes,  as well as the results in  Remark~\ref{pseudo-rem} and Proposition~\ref{type3-rem},   are shown in Figure~\ref{venn-fig}.
 We will define the classes of balanced and unbalanced split graphs in the next section. 
  
 %\begin{proof}
 %Since split graphs cannot have an induce $5$-cycle, NG-3 graphs are not split graphs.  Now suppose $G$ is an NG-1 or NG-2 graph.  Let its $ABC$-partition be $V(G) = A \cup B \cup C$.  If $G$ is an NG-1 graph,  
 
%By definition,  \ngthree graphs contain an induced $C_5$, hence are not perfect graphs, so they can not be split graphs.   Now suppose $G \in $ \NGonecuptwo \ and   let its $ABC$-partition be $V(G) = A \cup B \cup C$.  If $G\in$ \NGone,  let   $K = A \cup B$ and $S = C$  and otherwise, $G \in$ \NGtwo, and we  let    $K =  B$ and $S = A \cup C$.   In either case, using Theorem~\ref{NG-charac}, we get a $KS$-partition of $G$, so $G$ is a split graph.
%\end{proof}

%Thus  the class of pseudo-split graphs is not just almost  extremal in terms of the Nordhaus-Gaddum inequality,  it actually contains {\it all}  the NG-graphs.    

Building on  the work of   Bl\'{a}zsik et al. \cite{BlHuPlTu93},    Maffray and Preissmann \cite{MaPr94}   present a degree sequence characterization for NG-3 graphs. They combine this with the similar characterization for split graphs to get a linear-time recognition algorithm for  pseudo-split graphs. We will discuss similar algorithms for NG-graphs in Section~3.

Finally, we note that  Theorem \ref{NG-equal-thm} is very much related to the notion of graph decomposition that was studied systematically by Tyshchevich \cite{Ty00}.   A graph is said to be {\it decomposable} if its vertex set can be partitioned into three parts $A, B$ and $C$ so that $A\neq \emptyset$ and $B \cup C \neq \emptyset$ and conditions (ii) to (v) of Theorem \ref{NG-equal-thm} are satisfied.   %Otherwise, the graph is {\it indecomposable}.  
Thus,  every NG-graph is decomposable except when it is a single vertex or a 5-cycle.  Chv\'{a}tal and Hammer \cite{ChHa77}, Bl\'{a}zik et al. \cite{BlHuPlTu93} and Barrus \cite{Ba13, Ba14} also characterized various graph classes in terms of their decompositions.

 \medskip
 
 Our   main objective is to explore the connections between NG-graphs and split graphs. In Section 2, we study split graphs through the lens of NG-graphs.  In particular, we determine which split graphs are NG-graphs and show how their ABC-partitions relate to their clique-stable set partitions.   
% As a consequence,  we find that pseudo-split graphs are either NG-graphs or what we call {\it balanced split graphs}.  
  In Section 3,  we do the opposite and consider NG-graphs through the lens of split graphs.  We provide   degree characterizations for NG-1 and NG-2 graphs that are quite similar to the one for  split graphs.   We also present a degree characterization of NG-3 graphs that is equivalent to the one in \cite{MaPr94}. These  results show  that,  like split graphs and pseudo-split graphs,  NG-graphs have a linear-time recognition algorithm.   In Section 4 we present  various kinds of bijections including those between subclasses of NG-graphs and subclasses of split graphs.   Finally,  in Section 5,  we take advantage of these bijections and present formulas for the number of graphs on $n$ vertices in each of the graph classes we studied. Our work comparing NG-graphs to split graphs  leads to a theorem only about split graphs: that for $n\geq 1$, the number of unbalanced split graphs on $n$ vertices is equal to the number of split graphs on $n-1$ or fewer vertices. 
  % We also show that if the number of split graphs on $n$ vertices is growing rapidly then,   as $n \rightarrow \infty$, almost all split graphs are balanced.

\begin{figure}
\[\begin{picture}(400,120)
\thicklines
\put(150,60){\oval(300,120)}
\thinlines 
\put(250,80){\oval(300,120)}
\put(15,127){\bf NG-graphs}
\put(25,88){\NGthree}
\put(330,7){Split graphs}
\put(325,112){Balanced}
\put(325,100){split graphs}
\put(150, 88){\NGonecuptwo}
\put(142, 36){Unbalanced split graphs}

\end{picture}\]
\caption{A partition of the class of pseudo-split graphs into NG-graphs and split graphs.}
\label{venn-fig}
\end{figure}
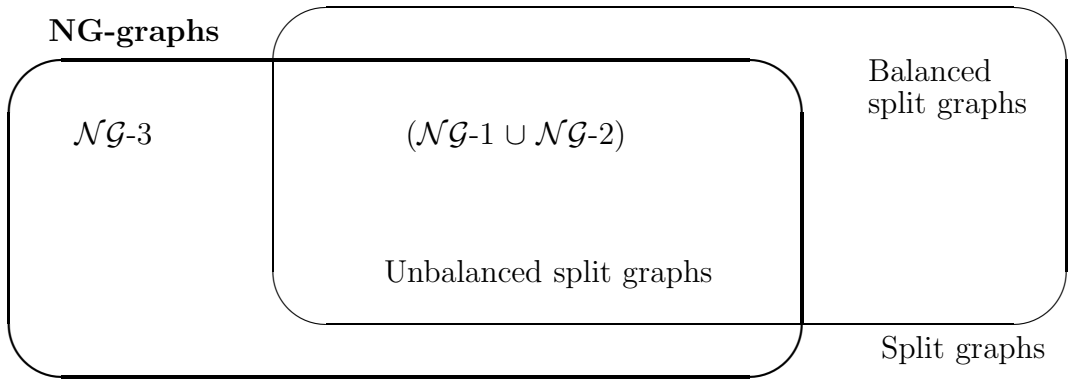

 \section{Split Graphs and NG-graphs}
 
 In this section we consider the set ${\cal S}$  of split graphs and discuss   connections to NG-graphs.  A $KS$-partition of a split graph $G$ is a partition of the vertex set as  $V(G)= K \cup S$ where $K$ is a clique and $S$ is a stable set.
 Just as it is helpful to characterize NG-graphs into the classes \NGone, \  \NGtwo \  and \NGthree \   based on their $ABC$-partition, it is also useful to categorize split graphs based on their $KS$-partitions.   

% \begin{Def}{\rm
%A \emph{split graph} is a graph $G$ whose vertex set can be partitioned as $V(G)   We call such a  partition   a .  We denote the set of split graphs  by ${\cal S}$.}
%\label{split-def}
%\end{Def}
%Proposition \ref{type3-rem} describes which NG-graphs are split graphs.  

% In the next result we describe which NG-graphs are split graphs.

%\begin{Prop} Let $G$ be an NG-graph.  Then $G$ is a split graph if and only if 
 % $G \in $ \NGonecuptwo.
%\label{type3-rem}
%\end{Prop}

%\begin{proof}
%By definition,  \ngthree graphs contain an induced $C_5$, hence are not perfect graphs, so they can not be split graphs.   Now suppose $G \in $ \NGonecuptwo \ and   let its $ABC$-partition be $V(G) = A \cup B \cup C$.  If $G\in$ \NGone,  let   $K = A \cup B$ and $S = C$  and otherwise, $G \in$ \NGtwo, and we  let    $K =  B$ and $S = A \cup C$.   In either case, using Theorem~\ref{NG-charac}, we get a $KS$-partition of $G$, so $G$ is a split graph. \end{proof}

 \begin{Def} {\rm
 A split graph $G$  is \emph{balanced} if it has a $KS$-partition satisfying $|K| = \omega (G)$ and $|S| = \alpha(G)$   and \emph{unbalanced} otherwise.    We denote the set of balanced split graphs  by $\cal B$ and the set of unbalanced split graphs by ${\cal U}$.  A $KS$-partition is \emph{$S$-max} if $|S| = \alpha(G)$    and \emph{$K$-max} if $|K| = \omega (G)$.     }
 \label{bal-unbal-def}
 \end{Def}
 
  Unlike $ABC$-partitions, $KS$-partitions of a   split graph are not always unique. The terms  \emph{balanced} and \emph{unbalanced} in Definition~\ref{bal-unbal-def} refer to a split graph $G$ while the terms \emph{$K$-max} and \emph{$S$-max} refer to a particular $KS$-partition of $G$. The next theorem  follows from the work of Hammer and Simeone \cite{HaSi81} and appears in \cite{Go80}.  
  We include a proof for completeness.
  
  \bigskip
  
%\marginpar{ Ann put back the proof.  }
 
 \begin{Thm} {\rm (Hammer and Simeone \cite{HaSi81})}
 For any $KS$-partition of a split graph $G$, exactly one of the following holds:
 
 (i)  $|K| = \omega(G)$ and $|S| = \alpha(G)$.  \hfill (balanced)
 
 (ii)  $|K| = \omega(G)-1$ and $|S| = \alpha(G)$. \hfill (unbalanced, $S$-max)

(iii)  $|K| = \omega(G)$ and $|S| = \alpha(G)-1$. \hfill (unbalanced, $K$-max)

\smallskip

Moreover, in (ii) there exists $s \in S$ so that $K \cup \{s\}$ is complete and in  (iii) there exists $k \in K$ so that $S \cup \{k\}$ is a stable set.

\label{split-thm}
 \end{Thm}
 
 \begin{proof}
 Partition the vertex set of $G$ as $V(G) = K \cup S$ where $K$ is a clique and $S$ is a stable set.  If both $K$ and $S$ are maximum size then $|K| =    \omega(G)$ and $|S| = \alpha(G)$, resulting in case (i).  If $K$ is not maximum size, then $\omega(G) = |K| + 1$ because only one vertex of $S$ can be part of a clique.  In this case,  there must exist $s \in S$   adjacent to each vertex in $K$.    Then  no vertex of $K$ can be added to $S$ to make a larger stable set, and at most one vertex of $K$ can be in any stable set, so $S$ is maximum size.  Thus $|S| = \alpha(G)$, resulting in case (ii), and moreover, $K \cup \{s\}$ is complete.   Similarly, if $S$ is not maximum size, the result is case (iii).   
 \end{proof}

In Proposition~\ref{type3-rem}, we showed which NG-graphs are split graphs.  In the next theorem we show which split graphs are NG-graphs.
 \begin{Thm}  
 \label{NG-unbalanced}
 The following are equivalent for    a split graph $G$. \\
 (1)  $G$ is an NG-graph.\\
 (2)  $G \in$ \NGonecuptwo. \\
 (3)  $G$ is unbalanced.
   \label{unbal-thm}
\end{Thm} 
 
\begin{proof}
\noindent
(1) $\Longrightarrow$ (2).  Follows directly from Proposition~\ref{type3-rem}.  \\

\noindent
(2) $\Longrightarrow$ (3). For a contradiction, assume $G$ is 
  a balanced split graph and fix a  a $KS$-partition with $|K| = \omega(G)$ and $|S| = \alpha(G)$.  Since split graphs are perfect,  $\chi(G) = \omega(G)$ and $\chi(\gbar) = \alpha(G)$. Thus $\chi(G) + \chi(\gbar) = \omega(G) + \alpha(G) = |K| + |S| = |V(G)| \neq  |V(G)|+1$ and $G$ is not an NG-graph.\\
  
\noindent
(3) $\Longrightarrow$ (1).   Let $G$ be an unbalanced split graph and fix  a   $KS$-partition of $G$.   First consider the case in which the $KS$-partition is $K$-max, thus  $|K| = \omega(G)$ and by Theorem~\ref{split-thm}, $|S| = \alpha(G) -1$.  Again, since split graphs are perfect,
so $\chi(G) = \omega(G)$ and $\chi(\gbar) =   \alpha(G)$.    
Then,
$\chi(G) + \chi(\gbar) = \omega(G) + \alpha(G) = |K| +  |S| +1= |V(G)|+1$ and $G$ is an NG-graph.    The proof for a $S$-max $KS$-partition is similar.    
 \end{proof}
 
The next remark follows from Proposition~\ref{type3-rem} and Theorem~\ref{unbal-thm}.  The Venn diagram in Figure~\ref{venn-fig} shows the relationships we have proven about NG-graphs and split graphs. 

\begin{remark}
  $\mathcal{U} =$ \NGonecuptwo, and consequently,  pseudo-split graphs are either NG-graphs or balanced split graphs.  
%The class of pseudo-split graphs is the disjoint union of \NGthree  and ${\cal B}$.  
\label{Un-rem}
\end{remark}

%\medskip

 Our knowledge of NG-graphs allows us to refine the Hammer/Simeone conditions in Theorem~\ref{split-thm}.
In particular, we can characterize all $KS$-partitions of a split graph and for split graphs with more than one $KS$-partition (unbalanced) it is precisely the vertices in $A_G$ that can be moved between $K$ and $S$.
%\begin{figure}[h] \label{Venn}
 
%\caption{The relationships between  \ngone, \ngtwo and \ngthree  graphs and balanced, unbalanced split graphs}
%\end{figure}  

 \begin{Thm}  
 Suppose $G$ is an unbalanced split graph  and let   $V(G) = A \cup B \cup C$ be its $ABC$-partition.   The $KS$-partitions of $G$ can be characterized as follows.  
 
 \begin{itemize}
\item If $G \in $ \NGone, the partitions are 
 
 $K = A \cup B$,  $S = C$ (unique $K$-max), 
 
 $K = (A \cup B) - \{a\}$, $S = C \cup \{a\}$ for any $a \in A$ ($S$-max).
 
\item If $G \in $ \NGtwo, the partitions are

    $K =  B$,  $S = A \cup C$ (unique $S$-max) 
    
     $K =   B \cup \{a\}$, $S = (A \cup C) - \{a\}$  for any $a \in A$ ($K$-max).
     
     \end{itemize}
  
 \label{num-part-thm}
 \end{Thm} 
 
 \begin{proof}
 Since $G$ is an unbalanced split graph, we know that $G$ is also an NG-graph by Theorem~\ref{NG-unbalanced}.   Indeed, by Proposition~\ref{type3-rem}, it is an \ngone graph or an \ngtwo graph. 
 We will give the argument in the case in which $G$ is an \ngone graph, that is, $G[A]$ is a clique.   The case in which $G$ is an \ngtwo  graph is similar.
 
  Let $K = A \cup B$ and $S = C$.  We claim that $|K| = \omega(G)$, that is, this is a $K$-max partition of $G$.    At most one vertex of stable set $C$ can be in any clique of $G$ but no vertex of $C$ can be added to $K$ because vertices of $C$ are not adjacent to vertices in $A$ and $A \neq \emptyset$.  Thus $|K| = \omega(G)$ as desired and by Theorem~\ref{split-thm}, $|S| = \alpha(G) -1$.
 
 We next show that this is the only $K$-max partition of $G$.  As before, at most one vertex of $C$ can be in a clique, so if there were a different $K$-max partition, we would have to move one vertex $c \in C$ from $S$ to the clique $K$ and remove one vertex from $K = A \cup B$ to $S$.  Since vertices in $C$ are not adjacent to vertices in $A$, we must remove all vertices in $A$ from $K$ in order to retain a clique, thus $|A| = 1$.    But in this case, the vertex $c$ and the vertex $a \in A$ have the same neighbor set (namely, all vertices in $B$) and thus the same degree, violating Definition~\ref{ABC-def}.  
 
 Next we  characterize the  $S$-max partitions.  Take any $a \in A$ and let $K' = (A   \cup B)-  \{a\} $ and  $S' = C \cup \{a\}$.  Thus $V(G) = K' \cup S'$.  We know $K'$ is a clique and $S'$ is a stable set by Theorem~\ref{NG-equal-thm}.  Furthermore,    $|S'| = |S| + 1 = \alpha(G)$ so the partition $V(G) = K' \cup S'$ is $S$-max for each $a \in A$.    Finally, we show these are the \emph{only} $S$-max partitions of $G$.  Since $A \cup B$ is a clique in $G$, at most one vertex of $A \cup B$ can be in any stable set.  If a vertex $b \in B$ were added to $S$ to form a larger stable set, then $b$ would have the same degree (namely $|A| + |B| -1$) as each vertex in $A$, violating Definition~\ref{ABC-def}. 
 \end{proof}
 
 The next corollary follows directly from Theorem~\ref{num-part-thm} and allows us to conclude that there are exactly $(|A| + 1)$ $KS$-partitions of an unbalanced labeled split graph. In contrast, there is a unique $KS$-partition of a balanced split graph as shown in Proposition~\ref{unique-bal-prop}.
 
 \begin{Cor}
 Suppose $G$ is an unbalanced split graph and let $V(G) = A \cup B \cup C$ be its $ABC$-partition.   Then if we consider $G$ to be a labelled graph, it has either a unique $K$-max partition and exactly $|A|$ distinct $S$-max partitions or a unique $S$-max partition and exactly $|A|$ distinct $K$-max partitions.
 \end{Cor}

 \begin{Prop} \label{unique-bal-prop} 
 Each balanced split graph has a unique $KS$-partition.
 \end{Prop}

  \begin{proof}
  Let $G$ be a balanced split graph and fix a $KS$-partition of $G$ with $|K| = \omega(G)$ and $|S| = \alpha(G)$.  For a contradiction, suppose  $K'S'$ is a different $KS$-partition of $G$.  Since $K$ and $S$ are already of maximum size and at most one vertex of $S$ can be in a clique and at most one vertex of $K$ can be in a stable set, we know $K' = K \cup \{y\} - \{x\}$ and $S' = S \cup \{x\} - \{y\}$ for some $x \in K$ and $y \in S$.  Now $K$ and $K'$ are cliques but $K \cup \{y\}$ has $\omega(G) + 1$ vertices and is not a clique, so $xy \not\in E(G)$.  But then $x$ is not adjacent to any vertex in $S$, so $S \cup \{x\}$ is a stable set of size $\alpha(G) + 1$, a contradiction.
   \end{proof}
 
 %\begin{Def} Let $K_1S_1$ and $K_2S_2$ be $KS$-partitions of an unlabeled split graph $G$.  We say these partitions are 
%\emph{isomorphic} if there is an automorphism $\phi$ of $G$ for which $\phi(S_1)=S_2$ (and consequently, $\phi(K_1)=K_2$).  
%\end{Def}

 We use the term ``unlabeled graphs'' in the remainder of the paper to refer to 
  isomorphism classes of graphs as in \cite{We01}.

\begin{Cor} \label{two-unbal-prop} Each unlabeled unbalanced split graph has exactly two   $KS$-partitions, one is $K$-max and the other is $S$-max.
\end{Cor}

\begin{proof}
Let $G$ be an unbalanced split graph.  Fix a $KS$-partition of $G$.  If it is $S$-max, then we can move a vertex from $S$ to $K$, as in the proof of Theorem \ref{split-thm}, to obtain a different $KS$-partition of $G$ which is $K$-max.  Similarly, if it is $K$-max, we can move a vertex from $K$ to $S$ to obtain a different $KS$-partition of $G$ which is $S$-max.

Now suppose there are two non-isomorphic $K$-max partitions of $G$, $K_1S_1$, $K_2S_2$.  Thus $|K_1|=|K_2|=\omega(G)$, but $K_1\neq K_2$.  Since at most one vertex of $S_1=V(G)-K_1$ can be in a clique, there exists $x\in S_1$ and $y\in K_1$ such that $K_2=K_1+\{x\}-\{y\}$.  Thus, $N(x)=K_1-\{x,y\}=N(y)$.  Then there is an automorphism of $G$ obtained by switching vertices $x$ and $y$ and leaving the remaining vertices unchanged.  This contradicts our assumption that $K_1S_1$ and $K_2S_2$ are not isomorphic.  The proof for two non-isomorphic $S$-max partitions is similar.
\end{proof}

 \section{Degree sequence characterizations}
 
 We begin with the Hammer and Simeone result showing that    split graphs  can be recognized  solely from  their degree sequences. 
 The proof serves as a foundation for the proofs of Theorems~\ref{split-gr-prop}  and \ref{not-split-gr-prop}. 
 
% \marginpar{put back proof.  Fixed "furthermore" to original}
  
 \begin{Thm} {\rm  (Hammer and Simeone, \cite{HaSi81})}
 Let $G = (V,E)$ be a graph with degree sequence $d_1 \ge d_2 \ge \cdots \ge d_n$ and let $m = max\{i:d_i \ge i-1\}.$  Then $G$ is a split graph if and only if 
 \begin{equation}
  \sum_{i=1}^m d_i = m(m-1) + \sum_{i = m+1}^n d_i. 
  \label{degree-condition}
 \end{equation}
 
 Furthermore, if equality holds in (\ref{degree-condition}),  then $\omega(G) = m$.
 \label{ham-sim-thm}
 \end{Thm}
 
 \begin{proof}
 Let $V(G)= \{v_1,v_2, \ldots, v_n\}$ where $deg(v_i) = d_i$ for each $i$.  Since $d_1 \ge 0$, the value $m$ is well-defined.  If $G=K_n$ then $m=n$, $G$ is a split graph,  equality (\ref{degree-condition}) holds  and $\omega(G) = m$ as desired.  Thus we may assume $G$ is not complete, thus $d_n < n-1$ and $m \le n-1$.  
 
 First  we assume that $G$ satisfies the equality (\ref{degree-condition}) in Theorem~\ref{ham-sim-thm}.  We let $K= \{v_1,v_2, \ldots v_m\}$ and $S = \{v_{m+1}, v_{m+2}, \ldots v_n\}$ and will show that $K$ is  a clique and $S$ is a stable set in $G$. Partition $E(G) = E_1 \cup E_2 \cup E_3$ where $E_1 = E(G[K])$,     $E_2 = \{xy:x \in K, y\in S\}$, and $E_3 = E(G[S])$.  Then $\dsp \sum_{i=1}^m d_i = 2|E_1| + |E_2|$ and $\dsp\sum_{i=m+1}^n d_i = 2|E_3| + |E_2|$.  Using equality (\ref{degree-condition}) we get $2|E_1| + |E_2| = m(m-1) + 2|E_3| + |E_2|$ or equivalently, $|E_1| = \frac{m(m-1)}{2} + |E_3|$.  Thus $|E_1| \ge \frac{m(m-1)}{2}.$  But $E_1$ is the edge set of the graph $G[K]$ with $m$ vertices, thus $|E_1| \le \frac{m(m-1)}{2}$ and we conclude $|E_1| = \frac{m(m-1)}{2}$ and $|E_3| = 0$.  Thus $K$ is a clique and $S$ a stable set as desired.
 
 Conversely, let $G$ be a split graph and by Theorem~\ref{split-thm}  we can fix a K-max partition of $G$.   As before, let $E_1$ be the edge set of $G[K]$, $E_2$ be the set of edges with one endpoint in $K$ and the other in $S$, and $E_3$ be the edge set of $G[S]$.  Since $K$ is a clique, $deg(v) \ge |K| -1$ for each $v \in K$ and since $S$ is a stable set and $K$ is a maximum size clique, $deg(v) \le |K| -1$ for each $v \in S$.  Thus $m=k$ and  without loss of generality, $K = \{v_1,v_2, \ldots, v_{|K|}\}$ and $S = \{v_{|K|+1}, v_{|K|+2}, \ldots, v_n\}$.  Now $\dsp \sum_{i=1}^m d_i = 2|E_1| + |E_2| = m(m-1) +|E_2|$ and 
 $\dsp \sum_{i=m+1}^n d_i = 2|E_3| + |E_2| = |E_2|$ since $K$ is a clique and $S$ a stable set.  Thus equality holds in $(\ref{degree-condition})$.
\end{proof}

\begin{remark}
\label{k-max-part}
In the proof of Theorem~\ref{ham-sim-thm}, the partition $K = \{v_1,v_2, \ldots, v_m\}$, $S = \{v_{m+1},v_{m+2}, \ldots, v_n\}$ is a $K$-max partition of $G$.
\end{remark}
 
We will see  in Examples \ref{type1-ex} and \ref{type2-ex} that  index $m$ is either the first or last index $i$ for which $d_i=m-1$.  The next theorem shows that this is true for all split graphs.  
 
 %While Theorem~\ref{ham-sim-thm} characterizes split graphs using their degree sequence, the following result allows us to rule out some graphs that are not split graphs by identifying $m$ and listing the degree sequence (see Example~\ref{type-3-ex}).
 
 \begin{Thm}   Let $G = (V,E)$ be a   graph with degree sequence $d_1 \ge d_2 \ge \cdots \ge d_n$ and  $m = max\{i:d_i \ge i-1\}.$ 
 If $d_{m-1} = d_m = d_{m+1}$ then $G$ is not a split graph. \label{middle-cor}
 \end{Thm}
 
 \begin{proof}
 For a contradiction, assume $G$ is a split graph.    If $d_m > m-1$ then by definition of $m$ we have $d_{m+1} \le m-1$ and we contradict $d_m = d_{m+1}$.  Otherwise, $d_m = m-1$.   Let $V(G) = \{v_1,v_2, \ldots, v_n\}$ where $deg(v_i) = d_i$ for each $i$.   As in the proof of 
 Theorem~\ref {ham-sim-thm}, vertices $v_1,v_2, \ldots, v_m$ form a maximum clique $K$ and vertices $v_{m+1}, v_{m+2}, \ldots, v_n$ form a stable set $S$.  Since $S$ is a stable set, $nbh(v_{m+1}) \subseteq K$ and $|nbh(v_{m+1})| = d_{m+1} = m-1$ so $v_{m+1}$ must be adjacent to $m-1$ vertices of $K$.  Thus $v_{m+1}$ is adjacent to at least one of $v_m,v_{m-1}$.  But vertices $v_m$ and $v_{m-1}$ each have degree $m-1$ and each already has $m-1$ neighbors in $K$, a contradiction.
\end{proof}
 
In Theorems~\ref{split-gr-prop} and \ref{not-split-gr-prop}, 
we characterize the graphs in \NGone, \NGtwo, and \NGthree \   by their degree sequences.

 \begin{Thm}
 Let $G = (V,E)$ be a   graph with degree sequence $d_1 \ge d_2 \ge \cdots \ge d_n$ and  $m = max\{i:d_i \ge i-1\}.$   Then $G\in$(\NGone)$_n$   if and only if  it satisfies
 
\begin{itemize}
\item[$(1)$]
 $\  \sum_{i=1}^m d_i = m(m-1) + \sum_{i = m+1}^n d_i$, and 
 \item[$(2)$] \ $d_m = m-1$ and $m$ is the largest index for which $d_i=m-1$.
 \end{itemize}
Similarly, $G\in $(\NGtwo)$_n$  if and only if it satisfies $(1)$ and 

\vspace{.03in}
\hspace{-.15in}$(2')$ $d_m=m-1$ and $m$ is the smallest  index $i$ for which $d_i = m-1$.
 
 %Furthermore, if $G$ is an \ngone graph or an \ngtwo graph, then $m$ will be either the largest or smallest index $i$ for which $d_i = m-1$.  If $m$ is the largest, then $G$ is an \ngone graph, and if it is the smallest then $G$ is an \ngtwo graph.
  \label{split-gr-prop}
  \end{Thm}
 
 \begin{proof}  Let $V(G) = \{v_1, v_2, \ldots, v_n\}$ where $deg(v_i) = d_i$ and let $A\cup B\cup C$ be the ABC-partition of $G$ First, we show that in both directions of the proof, $G$ must be a split graph.  
If $G\in$ \NGonecuptwo, then $G$ is a split graph by Proposition~\ref{type3-rem}.  Conversely, if $G$ satisfies (1), it is a split graph by Theorem~\ref{ham-sim-thm}.   Split graphs are perfect, so $\chi(G) = \omega(G)$.  

Let $K=\{v_1, v_2, \ldots, v_m\}$ and  $S=  \{v_{m+1}, v_{m+2}, \ldots, v_n\} $.  By   Theorem~\ref{ham-sim-thm} and Remark~\ref{k-max-part}, we know $K$ is a clique, $S$ is a stable set and $KS$ is a $K$-max partition of $G$.  Thus $\omega(G) = |K| = m$ and $\chi(G) = m$.  

For the forward direction, assume $G\in$ \NGonecuptwo. Then $G$ is an unbalanced split graph by Theorem~\ref{unbal-thm} and thus satisfies $(1)$ of Theorem~\ref{ham-sim-thm}. It remains to show that $G\in $ \NGone \ implies $(2)$ and  $G\in $ \NGtwo\  implies $(2')$. 

First assume $G\in$ \NGone. Then by Theorem~\ref{num-part-thm}, the partition $K=A\cup B$, $S=C$ is the unique $K$-max partition of $G$. Since $A\neq \emptyset$ for NG-graphs and $v_m$ has the smallest degree among vertices in $K$, we know that $v_m\in A$ and thus $d_m=\chi(G)-1=m-1$. Moreover, $v_{m+1}\in C$, thus by Definition~\ref{ABC-def}, $d_{m+1}<\chi(G)-1 = m-1$, proving $(2)$.

Now assume $G\in$ \NGtwo. Then by Theorem~\ref{num-part-thm}, any $K$-max partition of $G$ has the form $K= B\cup \{a\}$, $S=(A\cup C)-\{a\}$ for some $a\in A$. By Definition~\ref{ABC-def},
 $deg(a)=\chi(G)-1=m-1$ and $deg(b)>m-1$ for all $b\in B$. Hence $a$ is the unique vertex of smallest degree in $K$. Thus, $v_m=a$, and $deg(v_m) =m-1$ and $deg(v_{m-1})>m-1$, proving  $(2')$. 

Conversely,  assume $(1)$ holds. 
Since $\chi(G)=m$, $A$ is the set of vertices of degree $m-1$ in $G$. By Theorem~\ref{middle-cor}, $m$ is either the largest index for which $d_i=m-1$ or the smallest (or both).  

In the first instance, where $(2)$ holds, each of the vertices in $S$ has degree at most $m-2$ and $m-2  <\chi(G)-1$.  Hence $C =S $  and $A\cup B=K$. Each vertex in $A$ is adjacent to the other $m-1$ vertices in $K$, thus is not adjacent to any vertex in $C$, which proves that $G\in$ \NGone.
  
In the second instance, where $(2')$ holds, each of the vertices $v_1, v_2, \ldots, v_{m-1}$ has degree at least $m$ and $m>\chi(G)-1$.   Thus $B = \{v_1, v_2, \ldots, v_{m-1}\} $ and $B$ is contained in the clique $K$.  Since $K$ is a clique and $deg(v_m) = m-1$, vertex $v_m$ is adjacent to each of the vertices in $B$ and thus to no vertices in $S$.  Each $v\in S$ with $deg(v)=m-1$ is likewise adjacent to each vertex in $B$. Hence $A \cup C=S\cup \{v_m\}$ and  is a stable set. Thus, $G\in$ \NGtwo.
\end{proof}

The next two examples illustrate Theorem~\ref{split-gr-prop}.
 
 \begin{example} \label{type1-ex}
Let $G_1$ be the  \ngone graph whose ABC-partition is $A=\{v_4, v_5, v_6\}$, $B=\{v_1, v_2, v_3\}$ and $C=\{v_7, v_8, v_9, v_{10}\}$, and where the edge set between $B $ and $C$ is 
$E_2=\{ v_1v_7, v_1v_8, v_1v_9, v_2v_7, v_2v_8, v_3v_{10} \}$.   The table below shows the vertex degrees $d_i = deg(v_i)$.  Then  $m=6$ and as in Theorem \ref{split-gr-prop}, 6 is the largest index $i$ for which $d_i=m-1$.

\bigskip

\begin{center}
\begin{tabular}{|c|cccccc|cccc|} \hline 
$i$ & 1 & 2 & 3 & 4 & 5 & \underline{6} & 7 & 8 & 9 & 10  \\ \hline
$d_i$ & 8 & 7  & 6 & {\bf 5} & {\bf 5} & \underline{{\bf 5}} & 2 & 2 & 1 & 1 \\ \hline
$i-1$ & 0 & 1 & 2 & 3 & 4 & \underline{5} & 6 & 7  & 8 & 9  \\ \hline 
\end{tabular}
\bigskip

\end{center}

\end{example}

 \begin{example} \label{type2-ex} 
Let $G_2$ be the  \ngtwo graph formed by removing the 3 edges from $G_1$ between vertices in $A$.  Again, the table below shows the vertex degrees.  Now $m=4$ and indeed, $i=4$ is the smallest index $i$ for which $d_i=m-1$.

 \bigskip

\begin{center}
%\medskip

\begin{tabular}{|c|cccc|cccccc|} \hline 
$i$ & 1 & 2 & 3 & \underline{4} & 5 & 6 & 7 & 8 & 9 & 10  \\ \hline
$d_i$ & 8 & 7  & 6 & \underline{{\bf 3} }& {\bf 3} & {\bf 3} & 2 & 2 & 1 & 1 \\ \hline
$i-1$ & 0 & 1 & 2 & \underline{3} & 4 & 5 & 6 & 7  & 8 & 9  \\ \hline 
\end{tabular}

\end{center}

\bigskip
\end{example}

\bigskip

The next corollary follows directly from Theorems~\ref{ham-sim-thm},  \ref{middle-cor} and \ref{split-gr-prop}.
\begin{Cor}\label{numerical} $G\in$ \NGonecuptwo \ iff $G$ satisfies $(1)$ of Theorem~\ref{split-gr-prop} and $d_m=m-1$.
\end{Cor}

The following characterization of the class of balanced split graphs follows immediately from Theorem \ref{ham-sim-thm} and Corollary~\ref{numerical}.  Observe that Corollary~\ref{bal-recog-cor} does not make any reference to NG-graphs.

\begin{Cor}
 Let $G = (V,E)$ be a   graph with degree sequence $d_1 \ge d_2 \ge \cdots \ge d_n$ and  $m = max\{i:d_i \ge i-1\}.$   Then $G$ is a balanced split  graph   if and only if  it satisfies
 
\begin{itemize}
\item[(i)]
 $ \sum_{i=1}^m d_i = m(m-1) + \sum_{i = m+1}^n d_i$
 \item[(ii)] $d_m > m-1$.
 \end{itemize}
 %and $m$ is the largest index $i$ for which $d_i = m-1$.  
\label{bal-recog-cor}
\end{Cor}

In the next theorem, we characterize the class \NGthree \ by degree sequences.   In this instance,  $i$ is the middle index of five in which $d_i = m-1$.   An equivalent result appears  in \cite{MaPr94}.

 \begin{Thm} 
 Let $G = (V, E)$ be a  graph with degree sequence $d_1 \ge d_2 \ge \cdots \ge d_n$ and let $m = max\{i:d_i \ge i-1\}.$   Then   $G\in$(\NGthree)$_n$ if and only if  the following conditions hold:
 \begin{itemize}
  \item[(i)]  $  \sum_{i=1}^{m+2} d_i = (m+2)(m+1) - 10 + \sum_{i= m+3}^n d_i$
 
 \item[(ii)]  $d_i = m-1$ if and only if $m-2 \le i \le m+2$

 \end{itemize}

 \label{not-split-gr-prop}  
   \end{Thm}
  
  \begin{proof}
  Let $V(G) = \{v_1, v_2, \ldots, v_n\}$ where $deg(v_i) = d_i$.  First we suppose that $G$ satisfies the two conditions and show $G$ is an \ngthree graph.  Partition $V(G)$ into sets $B$, $A$, $C$ as follows and let  $X = A \cup B$:
  
   $B = \{v_1, v_2, \ldots, v_{m-3}\}$, 
   
   $A = \{v_{m-2}, v_{m-1}, v_m, v_{m+1}, v_{m+2}\}$,
   
    $C = \{v_{m+3}, v_{m+4}, \ldots, v_n\}.$ 
     
     \smallskip
     
     We will show that this is the ABC-partition of $G$.  Partition the edge set of $G$ as $E_1 \cup E_2 \cup E_3$ where  $E_1$ is the edge set of $G[X]$, $E_3$ is the edge set of $G[C]$ and $E_2$ is the set of edges in $G$ with one endpoint in $X$ and the other in $C$.   Summing vertex degrees in $X$ and in $C$ we get $$ \sum_{i=1}^{m+2} d_i = 2|E_1| + |E_2| \hbox {  and  } \sum_{i=m+3}^{n} d_i = 2|E_3| + |E_2|.$$  Using   condition (i) in the hypothesis, we get $$2|E_1| + |E_2| = (m+2)(m+1) - 10 + 2|E_3| + |E_2|,$$ or equivalently,
  \begin{equation}
    2|E_1|   = (m+2)(m+1) - 10 + 2|E_3|.   
  \label{eq-2}
  \end{equation}
  Since $|X| = m+2$, each vertex in $X$ has degree in $G[X]$ of at most $m+1$, and by   condition (ii), the five vertices in $A$ have degree $m-1$ in $G$, thus their degree is at most $m-1$ in $G[X]$.  So summing degrees of vertices in the induced graph $G[X]$ we get 
  \begin{eqnarray*}
  2|E_1| &=& \sum_{i=1}^{m+2} deg_{G[X]}(v_i)\\
  & \le& 5(m-1) + (m-3)(m+1)   = 5 (m+1) - 10 + (m-3)(m+1)\\
  &=&    (m+2)(m+1) - 10.
  \end{eqnarray*} 
  Combining this with equation~(\ref{eq-2}) we get $|E_3| = 0$ and conclude that $C$ is a stable set.   Now using $|E_3| = 0$ in (\ref{eq-2}) we get     $|E_1| = \binom{m+2}{2} - 5$, so $G[X]$ is a complete graph on $m+2$ vertices with five edges removed.
  
  However, $G[X]$ contains five vertices of degree $m-1$, namely those in $A$.  Thus each vertex in $A$ must be incident to exactly two of the removed edges, and the set of removed edges forms a 5-cycle.  The edges remaining in $G[A]$ also form a 5-cycle (since $\overline{C_5} = C_5$).  Thus $G[X]$ consists of a clique $G[B] \approx K_{m-3}$, a 5-cycle $G[A]$ and all edges between them.    We have already verified conditions (i) -- (iv) of Theorem~\ref{NG-equal-thm}.   Since $G[A]$ is a 5-cycle, any  largest clique in $G[X]$ consists of the vertices in $B$ together with two adjacent vertices of $A$, thus $\omega(G) \ge \omega(G[X]) = m-1$.   Each vertex in $A$ has degree $m-1$ in $G$ with $m-3$ neighbors in $G[B]$ and 2 neighbors in $G[A]$, thus   vertices  in $A$ are not adjacent to any vertices of $C$, proving (v).     This  also means that vertices in $C$ have degree at most $|B| = m-3$ and  thus cannot participate in a clique of size $m$, hence $\omega(G) = m-1$. 
  
   However, $\chi(G) \ge m$ since the 5-cycle $G[A]$ requires three colors and an additional $m-3$ colors are needed for the clique $G[B]$.    This coloring can be extended to all of $V(G)$ by assigning a color used in $A$ to all vertices of $C$.  Thus $\chi(G) = m$, the vertices in   $A$ have degree $m -1$, vertices in $B$ have degree greater than $m-1$ and vertices in $C$ have degree less than $m-1$, so the partition $V(G) = A \cup B \cup C$ is in fact an ABC-partition of $G$.  By Theorem~\ref{NG-equal-thm}, $G$ is an NG-graph and because $G[A]$ is a 5-cycle, it is an \ngthree  graph.

Conversely, suppose $G$ is an \ngthree graph, so its ABC-partition satisfies the conditions of Theorem~\ref{NG-charac} with $G[A]$ a 5-cycle.  By the definition of an ABC-partition, vertices in $B$ have degree greater than vertices in $A$ which in turn have degree greater than those in $C$.  Thus $B = \{v_1, v_2, \ldots, v_{|B|}\}$, $A = \{v_{|B|+1}, v_{|B|+2}, \ldots, v_{|B|+5}\}$,  and $C = \{v_{|B|+6},   \ldots, v_n\}$.  By the structure of \ngthree graphs, each vertex  in $A$ has degree $2+|B|$, thus $deg(v_{|B|+3}) \ge |B| + 2$ but $deg(v_{|B|+4}) < |B| + 3$ and hence $m = |B| + 3$.  Therefore, 

$B = \{v_1, v_2, \ldots, v_{m-3}\}$, 
   
   $A = \{v_{m-2}, v_{m-1}, v_m, v_{m+1}, v_{m+2}\}$,
   
    $C = \{v_{m+3}, v_{m+4}, \ldots, v_n\}$.

    Vertices in $A$ have degree $2 + |B| = m-1$ and these are the only vertices of degree $m-1$ by definition of the ABC-partition.  Thus the second condition in Theorem~\ref{not-split-gr-prop}  holds.
Vertices in $C$ can only have neighbors in $B$, thus $deg(c) \le |B| = m-3$ and we conclude that $d_i = m-1$ if and only if $m-2 \le i \le m+2$, which is the second condition of Proposition~\ref{not-split-gr-prop}.

     Finally, let $E_1$ be the edge set of $G[A\cup B]$ and $E_2$ be the set of edges in $G$ with one endpoint in $B$ and the other in $C$.  Since $G[A\cup B]$ is a clique with a 5-cycle removed we have,  $$ \sum_{i=1}^{m+2} d_i = 2|E_1| + |E_2|  = (m+2)(m+1) - 10 + |E_2|.$$
      Also,  because $C$ is a stable set and there are no edges between $A$ and $C$, $ \sum_{i=m+3}^{n} d_i =  |E_2|$.  Combining these gives the first condition   in Proposition~\ref{not-split-gr-prop}.
    \end{proof}
    
\begin{example}
Let $G_3$ be the \ngthree graph formed from $G_1$ in Example \ref{type1-ex} by expanding $A$ to be a 5-cycle.  Thus, $m=6$ which is the middle index of the 5 vertices of degree $d_m=5$.  
\medskip

\begin{center}

\begin{tabular}{|c|cccccccccccc|} \hline 
$i$ & 1 & 2 & 3 & 4 & 5 & \underline{6} & 7 & 8 & 9 & 10 & 11 & 12 \\ \hline
$d_i$ & 10 & 9 & 8 & {\bf 5} & {\bf 5} & \underline{{\bf 5}} & {\bf 5} & {\bf 5} & 2 & 2 & 1 & 1 \\ \hline
$i-1$ & 0 & 1 & 2 & 3 & 4 & \underline{5} & 6 & 7  & 8 & 9 & 10 & 11 \\ \hline 
\end{tabular}
\medskip 

\end{center}

\label{type3-ex}
\end{example}
    
We know that split graphs are not \ngthree graphs by Proposition \ref{type3-rem}.   Theorem~\ref{middle-cor} shows that the second condition of Theorem~\ref{not-split-gr-prop} is never  satisfied for a split graph.  The following example shows that it is possible for a split graph to satisfy the first condition of Theorem~\ref{not-split-gr-prop}.  

\begin{example}
Let $G$ be the split graph with $KS$ partition $K=\{v_1, v_2, v_3\}$,   $S=\{v_4, v_5, v_6\}$ and $E(G) = E(K) \cup \{v_1v_3,v_2v_4\}$.   The vertex degrees are:  $d_1=3, d_2=3, d_3=2, d_4=1, d_5=1, d_6=0$, so  $m=3$.  Notice that the first condition in Theorem \ref{not-split-gr-prop} is satisfied, but the second is not.  
\end{example}

Finally,  since sorting the degrees of the vertices of a graph can be done in linear time using stable sort,  Theorems~\ref{split-gr-prop} and \ref{not-split-gr-prop} imply the next corollary. 
    
%The recognition times given in the next Corollary follow  directly from Theorems~\ref{split-gr-prop} and \ref{not-split-gr-prop}.
    
    \begin{Cor}  Given a graph $G$,  determining if $G$ is an NG-1, NG-2 or an NG-3 graph can be done in linear time.   
    %    Each of the classes    \NGone, \NGtwo, \NGthree \ can be 
     % recognized in linear time if the input is a graph and its degree sequence. Membership in each class   is completely determined by degree sequence.
    \end{Cor}

\section{Bijections}
 \label{bijections-sec}

Throughout  this section we are considering  unlabeled graphs.    
 Definition~\ref{NG-123-def}  divides the set of NG-graphs into three categories:   \NGone,  \NGtwo,  and \NGthree, according to whether the set $A$ of the $ABC$-partition induces a clique, a stable set or a 5-cycle.  The first two categories overlap in the case where $|A| = 1$.  By removing this intersection as a separate class, we can partition the set of $NG$-graphs into four classes:  \NGonemtwo,  \NGtwomone, \NGonecaptwo \  and \NGthree. Recall that we also divide split graphs $(\cal S)$ into balanced $(\cal B)$ and unbalanced $(\cal U)$.  Let $({\cal T})_{\le n}$ be the set of split graphs on $n$ or fewer vertices and recall that 
    $(\mathcal{C})_n$  denotes the set of graphs with $n$ vertices in a class $\cal C$.  
 
 In this section we provide bijections between classes of NG-graphs and classes of split graphs.    In Section~\ref{counting}, we use these results to count the number of graphs in each class. 
 Table~\ref{bij-table} summarizes the bijection results in this section and Theorem~\ref{table-thm} records the locations of each of the proofs.  The proofs are organized by the type of bijection used.
   
\begin{Thm}\label{table-thm}
There is a bijection between any two classes of graphs that appear in the same row of Table~\ref{bij-table}.
\end{Thm}
\begin{proof}
The bijections between the classes of graphs in row (a) appear in Theorem~\ref{ngone-to-split}.  The bijection between (iv) and (iv') in the proof    of Theorem~\ref{ngone-to-split} gives the bijection between the classes in row (b).  Complementation provides a bijection between \NGonemtwo$_n$ and 
\NGtwomone$_n$ in row (c).  A bijection beween \NGonemtwo$_n$ and $({\cal U})_{n-1}$ in row (c) is obtained by combining Remark~\ref{Un-rem} with the bijection given in  the proof of Theorem~\ref{ngone-to-split} between
the union of classes (i), (ii) and (iii) and the union of classes (i'), (ii') and (iii').  For  row (d), Theorem~\ref{un-ngthree-thm}
 provides a bijection between (\NGthree)$_n$ and $({\cal U})_{n-4}$.    The remaining bijections appear in Theorem~\ref{glor-thm}.
\end{proof}

    \begin{table}[h]
 \begin{center}
 \begin{tabular}{|l|l|l|l|l|} \hline
 (a)& (\NGone)$_n$  & (\NGtwo)$_n$ & $({\cal S})_{n-1}$ &  \\ \hline
  (b)&  \NGonecaptwo$_n$& $({\cal B})_{n-1}$ &  &  \\ \hline
  (c)& \NGonemtwo$_n$ & \NGtwomone$_n$ &$({\cal U})_{n-1}$  & $({\cal T})_{\leq n-2}$ \\ \hline
  (d)& (\NGthree)$_n$  &$({\cal U})_{n-4}$    & $({\cal T})_{\leq n-5}$ &  \\ \hline
 \end{tabular}
 \end{center}
 
 \caption{There are bijections between all classes in the same row.}
 \label{bij-table}
 \end{table}

%\begin{Def} \label{s-t-u-b} 
%We use the following notation for sets of unlabeled split graphs.

%$({\cal S})_n=\{ \mbox{split graphs on $n$ vertices} \}$. 
 
%$({\cal T})_{n}=\{ \mbox{split graphs on  at most $n$ vertices} \}$. 
 
%$({\cal U})_n=\{ \mbox{unbalanced split graphs on $n$ vertices} \}$. 
 
%$({\cal B})_n=\{ \mbox{balanced split graphs on $n$ vertices} \}$. 
%\end{Def}

 We begin with a lemma that identifies the type of  $KS$-partition that results from removing $A_G$ from an NG-graph $G$.

\begin{lemma}\label{glorious-lem} Let $G$ be an NG-graph and $V(G) = A \cup B \cup C $ be its   ABC-partition.  Then $G-A$ is a split graph and has a  $KS$-partition with $K=B $ and $S=C$.  Moreover, 
\begin{enumerate}
\item if $G \in $ \NGone   \ then $KS$  is $S$-max, and 
%either $G-A$ is a balanced split graph or $G-A$ is an unbalanced split graph  and $K\cup S$ is  a $S$-max partition, and 
\item if $G \in $ \NGtwo \  then  $KS$  is $K$-max.
%or $G-A$ is an unbalanced split graph  and $K\cup S$ is  a $K$-max partition, and 
\end{enumerate}
\end{lemma}
\begin{proof} It is immediate that $G-A$ is a split graph and that $K=B$, $S=C$ constitutes a $KS$-partition of $G-A$.  In the ABC-partition of an \ngone graph $G$, every  vertex in $B$ has degree greater than every vertex in $A$ so each vertex in $B$ has a neighbor in $C$.  Thus in  our $KS$ partition of $G-A$, no vertex in $ K$ can  be moved to $ S$.  Hence, $\alpha(G-A)= |S| = |C|$. Similarly, in the ABC-partition of an \ngtwo graph $G$, every vertex in $C$ has degree less than any vertex in $A$, so every vertex of $C$ has a non-neighbor in $B$.  Hence in $G-A$, no vertex in $ S$ can be moved to $K$.  Thus, $\omega(G-A)= |K| = |B|$. 
\end{proof}

For an NG-graph $G$, let $C'_G$ be the set of vertices in $C_G$ whose neighbor set  is all of $B_G$.    In the proof of Theorem~\ref{ngone-to-split} we obtain our bijections by removing a single vertex from $A_G$.

\begin{Thm} For $n\geq 1$, there  are bijections between the following three classes: (\NGone)$_n$, (\NGtwo)$_n$ and
 $({\cal S})_{n-1}$. 
%\NGonemtwo$_n$    and \NGonecuptwo$_{n-1}$
\label{ngone-to-split}
\end{Thm}

\begin{proof}   Taking graph complements provides a bijection between the classes (\NGone)$_n$ and  (\NGtwo)$_n$, so it remains to show there exists a bijection between (\NGone)$_n$ and $({\cal S})_{n-1}$. 

 For $G \in $ (\NGone)$_n$, choose any $w \in A_G$ and let $f(G) = G-w$.  We will show that $f$ is our desired bijection.  Indeed, we will do more.   Partition the set of graphs in (\NGone)$_n$ into four classes:  (i) those $G$ with $|A_G| \ge 3$, (ii) those $G$ with $|A_G| = 2$ and $C'_G = \emptyset$, (iii) those $G$ with $|A_G| = 2$ and $C'_G \neq \emptyset$ and (iv) those $G$ with $|A_G|=1$.  
Likewise, partition the set of graphs in $({\cal S})_{n-1}$ into four classes:  (i') \NGonemtwo$_{n-1}$, (ii')  \NGonecaptwo$_{n-1}$, (iii') \NGtwomone$_{n-1}$ and (iv') $({\cal B})_{n-1}$.  We will show that $f$ is a bijection between each of (i), (ii), (iii) and (iv) and its corresponding class (i'), (ii'),  (iii'), (iv').  
 
 Let $G \in  $ \NGone$_n$  and let $A = A_G$, $B = B_G$,  $C = C_G$ and $C' = C'_G$, so $G[A]$ is a clique.  Then $\chi(G) = |A| + |B|$ so each $a \in A$ has $deg_G(a) = |A| + |B| - 1$, each $b \in B$ has $deg_G(b) > |A| + |B| - 1$, and each $c \in C$ has $deg_G(c) < |A| + |B| - 1$.  The graph $G-w$ has $\chi(G-w) = |A| + |B| -1$.   Each $a \in A-w$ has $deg_{G-w}(a) =  deg_G(a) - 1= |A| + |B| - 2$,  thus $a \in A_{G-w}$.  Each $b \in B$ has $deg_{G-w}(b)  = deg_B(b) -1 >  |A| + |B| - 2$ thus $b \in B_{G-w}$.  However, each $c \in C$ has $deg_{G-w}(c) = deg_G (c)$,
so some vertices in $C$ may be part of $A_{G-w}$.  We consider the cases mentioned above.
\medskip

\noindent {\bf Case (i):}  $|A_G| \ge 3$.  In this case, each $c \in C$ has   $deg_{G-w}(c) = deg_G(c) \le |B| < |A| + |B| -2 = \chi(G-w) - 1$, thus $c \in C_{G-w}$.  The ABC-partition of $G-w$ satisfies the five conditions of Theorem~\ref{NG-charac}, $A_{G-w}$ induces a clique, and $|A_{G-w}| \ge 2$, thus $G-w \in $ \NGonemtwo$_{n-1}$.
\smallskip

\noindent {\bf Cases (ii) and (iii):}  $|A_G| = 2$.   In this case, $A \cup C' -w$ forms a stable set and each vertex in $A \cup C' -w$ has neighbor set $B$.    Then for any $c' \in C'$ we have $deg_{G-w}(c') = |B| = (|A| -1) + |B| -1 = \chi(G-w) -1$ so $A_{G-w} = A \cup C' - w$, $B_{G-w} = B_G$, $C_{G-w} = C - C'$.  The ABC-partition of $G-w$ satisfies the five conditions of Theorem~\ref{NG-charac}, $A_{G-w}$ induces a stable set, thus $G-w \in $ (\NGtwo)$_{n-1}$.  If $C'_G = \emptyset$ (case (ii)) then $|A_{G-w}| =1$ and $G-w \in  $
\NGonecaptwo$_{n-1}$.  If $C'_G \neq \emptyset$ (case (iii)) then $|A_{G-w}| \ge 2$ and $G-w \in$ \NGtwomone$_{n-1}$.
\smallskip

\noindent {\bf Case (iv):} $|A|=1$. By Lemma~\ref{glorious-lem}, $G-A$ is a split graph on $n-1$ vertices with a $KS$-partition that is both $K$-max and $S$-max.  By Definition~\ref{bal-unbal-def}, the graph $G-A$ is a balanced split graph, thus $G-A \in ({\cal B})_{n-1}$.
\smallskip

Next, we show the map $f$ is onto by defining its inverse, $g$.  Take any $H \in ({\cal S})_{n-1}$.
\medskip

\noindent {\bf Case (i'):}   $H \in $ \NGonemtwo$_{n-1}$.  Thus $|A_H| \ge 2$ and $A_H$ induces a clique.  Add a vertex $w$ to $H$ that is adjacent to every vertex in $A_H \cup B_H$ to get $H+w$ and define $g(H) = H+w$.  

  Then $\chi(H) = |A_H| + |B_H|$ so each $a \in A_H$ has $deg_H(a) = |A_H| + |B_H| -1$,  each $b \in B_H$ has $deg_H(b) > |A_H| + |B_H| -1$, and each $c \in C_H$ has $deg_H(c) < |A_H| + |B_H| -1$.  
  The graph $H+w$ has $\chi(H+w) = |A_H| + |B_H|+ 1$.  Each   $a \in A_H \cup \{w\}$ has $deg_{H+w}(a) = |A_H| + |B_H|  $ thus $a \in A_{H+w}$.  Each   $b \in B_H  $ has $deg_{H+w}(b) = deg_H(b) + 1   $ thus $b \in B_{H+w}$.  Each   $c \in C_H  $ has $deg_{H+w}(c) = deg_H(c) $ thus $c \in C_{H+w}$.   The ABC-partition of $H+w$ satisfies the five conditions of Theorem~\ref{NG-charac},   $A_{H+w}$ induces a clique and $|A_{H+w}| \ge 3$, thus $H+w \in$   \NGonemtwo$_{n-1}$  (case i).
  \smallskip
  
\noindent {\bf   Cases (ii') and (iii'):}  $H \in $ (\NGtwo)$_{n-1}$.  In this case, $A_H$ is a stable set.  
 Let $x$ be a vertex in $A_H$. Now $\chi(H) =   |B_H| + 1$.   Each $a \in A_H$ has $deg_H(a) = \chi(H) -1 =    |B_H|  $,  each $b \in B_H$ has $deg_H(b) > \chi(H) -1 =   |B_H|  $, and  each $c \in C_H$ has $deg_H(c) <  \chi(H) -1 =  |B_H|$.   Add a vertex $w$ to $H$ that is adjacent to $x$ and all of $B_H$ to get $H+w$ and let $g(H) = H+w$.  Then $\chi(H+w) = \chi(H) + 1 = |B_H| + 2$.  The vertices $x$ and $w$ have $deg_{H+w}(x) = deg_{H+w}(w) = |B_H| + 1$, thus $x,w \in  A_{H+w}$.  Each $b \in B_H$ has $deg_{H+w}(b) =  deg_H(b) + 1 > |B_H| + 1$, thus $b \in  B_{H+w}$.  Each $c \in A_H \cup C_H - \{x\}$ has $deg_{H+w}(c) \le |B_H|  < \chi(H+w) - 1$, thus $c \in  C_{H+w}$.  The ABC-partition of $H+w$ satisfies the five conditions of Theorem~\ref {NG-charac},   $A_{H+w}$ induces a clique and $|A_{H+w}|  =2$, thus $H+w \in$  (\NGone)$_n$ with $|A|=2$ (cases ii and iii). 
  
 Indeed, if $H \in $  (\NGone \ $\cap$ \NGtwo) (case ii') then $A_H = \{x\}$ and no vertices in $H+w$ have degree $|B_H|$.  But $|B_H| = |B_{H+w}|$, so no vertices in $H+w$ have degree $ |B_{H+w}|$ and $C'_{H+w} = \emptyset$ (case ii).  Finally, if $H \in $ (\NGtwo \ $-$ \NGone) (case  iii') then $|A_H| \ge 2$ and there exists at least one vertex $y \in A_H - \{x\}$.  Then $deg_H(y) = deg_{H+w}(y) = |B_H| = |B_{H+w}| $  so $y \in  C'_{H+w}$ and $C'_{H+w} \neq \emptyset$ (case iii).
\smallskip
   
\noindent {\bf Case (iv'): } $H\in ({\cal B})_{n-1}$. In this case,  $H$ has a unique $KS$-partition by Proposition~\ref{unique-bal-prop} such that $|K|=\omega(H)=\chi(H)$ and $|S|=\alpha(H)$.  Add a vertex $w$ to $H$ that is adjacent to each vertex in $K$ and let $g(H) = H+w$.
Then $\chi(H+w) = |K| + 1$ and $deg(w) = |K| = \chi(H+w) -1$.  Each $x \in K$ has a  neighbor in $S$ since $|S| = \alpha(H)$, thus $deg_{H+w}(x) > |K| = \chi(H+w) -1$.  
Similarly, each $y \in S$ has a non-neighbor in $K$ since $|K| = \omega(H)$, thus $deg_{H+w}(y) <|K| = \chi(H+w) -1$.   
The ABC-partition of $H+w$ is $A = \{w\}$, $B = K$, $C = S$ and this satisfies the five conditions of Theorem~\ref {NG-charac},  with  $|A_{H+w}|  =1$, thus $H+w \in $    (\NGone)$_n$ with $|A|=1$ (case iv).    
\smallskip

It is not hard to see that the function $g$ is  the inverse  to $f$, thus we have described a bijection between (\NGone)$_n$ and $({\cal S})_{n-1}$.  
\end{proof}

  \begin{Thm}
  There is a bijection between (\NGthree)$_n$ and $({\cal U})_{n-4}$.
  \label{un-ngthree-thm}
   \end{Thm}

\begin{proof}  We will demonstrate a bijection between (\NGthree)$_n$ and \NGonemtwo$_{n-3}$.
  Then the desired bijection follows  because we have already established bijections between the first three classes in row (c) of Table~\ref{bij-table}.   
Let $G\in$ (\NGthree)$_n$ and let $V(G) = A\cup B\cup C$ be its ABC-partition. Then $G[A]=C_5$, by definition. Let $a_1, a_2\in A$ such that $\{a_1, a_2\}$ is an edge in $G[A]$, and define the map $f$ by $f(G)= G[\{a_1, a_2\} \cup B \cup C]$.  The graph $f(G)$ has $n-3$ vertices and is a split graph with $KS$-partition  $K=\{a_1, a_2\}\cup B$ and $S=C$. Since the neighborhood of $a_1$ does not include any vertices in $C$, then another $KS$-partition is $K'=B$ and $S=C\cup \{a_1\}$. By Proposition~\ref{unique-bal-prop}, $f(G)$ is an unbalanced split graph, hence $f(G) \in $ \NGonecuptwo.   Note that  $\omega(f(G))= \chi(f(G))=2+|B_G|$. The vertices in $f(G)$ that have degree $1+|B_G|$ include $a_1$ and $a_2$, and also any vertices in $B_G$ that have no neighbors in $C_G$.  Hence $f(G)\in $ \NGonemtwo$_{n-3}$

Now let $H\in$ \NGonemtwo$_{n-3}$
%Note that  $H$ is an NG-graph by Theorem~\ref{unbal-thm}.  Let    
with ABC-partition $A_H\cup B_H\cup C_H$ and let $a_1, a_2\in A_H$. We define the map $g $  as follows:  $g(H)$ is obtained by adding  three vertices $y_1, y_2, y_3$ to $H$ so that $H[\{a_1,a_2, y_1, y_2, y_3\}]$ is a 5-cycle and adding  all edges between $y_1, y_2, y_3$ and $A_H-\{a_1, a_2\}\cup B_H$. Then $\chi(g(H))= 3+|A_H|-2+|B_H|=|A_H|+|B_H|+ 1$. Each of $a_1, a_2, y_1, y_2, y_3$ has 2 neighbors in the 5-cycle and $|A_H|-2+|B_H|$ other neighbors, so these five vertices are in $A_{g[H]}$. The vertices in   $A_H-\{a_1, a_2\}\cup B_H$  form a clique and each has  at least $5+|A_H|-3+|B_H|= 2+ |A_H|+|B_H|$ neighbors, so these vertices are in $B_{g[H]}$.
Finally,  the vertices in $C_H$ form a stable set and each has at most $|B_H|$ neighbors, so these vertices are in $C_{g[H]}$. Hence the ABC-partition of $g(H)$ is $A_{g(H)} = \{a_1,a_2, y_1, y_2, y_3\}, B_{g(H)} = A_H-\{a_1, a_2\}\cup B_H, $ and $C_{g(H)}= C_H$, and $g(H)\in $ (\NGthree)$_n$.  It is straightforward to check that $g$ is the inverse function of $f$.
\end{proof}

%We will give bijections  between classes of $NG$-graphs and split graphs  in the proof of Theorem~\ref{glor-thm}.    In the first two parts, the mapping from an $NG$-graph $G$ to a split graph $H$ is obtained by removing the set $A_G$.  In the third part, this mapping is not injective, as seen in Example~\ref{non-inj-ex}.  Instead we remove the vertices in $A_G$ as well as the vertices in $B_G$ that have no neighbors in $C_G$.
	 
For an NG-graph  $G$, let $B'_G$ be the set of vertices in $B_G$ that have no neighbors in $C_G$.  	
	In the proof of Theorem~\ref{glor-thm},  we obtain our bijections by removing $A_G$ in part (1) and $A_G \cup B'_G$ in part (2).  In both cases, the result is a   split graph on $n-2$ or fewer vertices.  This motivates our defining 
   $({\cal T})_{\leq n}$ to be the set of split graphs on $n$ or fewer vertices and defining $|({\cal S})_0| = |({\cal T})_{\leq 0}| = 1$.  Then by definition,  
$({\cal S})_n=({\cal U})_n\cup ({\cal B})_n$ and $({\cal T})_{\leq n} = ({\cal S})_0 \cup ({\cal S})_1  \cup ({\cal S})_2 \cup \cdots \cup ({\cal S})_n$.

\begin{Thm} \label{glor-thm} Let $G$ be an unlabeled NG-graph on $n$ vertices.  %Fix the   ABC-partition of $G$. 
Then there are bijections between the following pairs of sets:
\begin{enumerate}
\item  \NGonemtwo$_n$ and  $({\cal T})_{\leq n-2}$.
%\item    \NGtwomone$_n$ and  $({\cal T})_{\leq n-2}$.
%\item   (\NGone \ $\cap$ \NGtwo)$_n$      and  $({\cal B})_{n-1}$.
\item  (\NGthree)$_n$   \   and  $({\cal T})_{\leq n-5}$.

 \end{enumerate}
\end{Thm}

\begin{proof}  Let $G$ be an NG-graph on $n$ vertices  and let $A\cup B\cup C$ be its ABC-partition. 

\medskip

\noindent {\bf Proof of (1):}  Let  $G\in $ \NGonemtwo$_n$.  By definition of an NG$-1$ graph, $deg(b)>deg(a)$
 for any $a\in A, b\in B$,  and hence every $b\in B$ has a neighbor in $C$. Define $f:$ \NGonemtwo$_n\rightarrow ({\cal T})_{\leq n-2}$ by $f(G) = G[B\cup C]$.  We will show $f$ is our desired bijection.  Since $|A| \ge 2$, the graph $G[B\cup C]$ has at most $n-2$ vertices, and by Lemma~\ref{glorious-lem},  it is a split graph,  thus $G[B\cup C] \in ({\cal T})_{\leq n-2}$. Since every $b\in B$ has a neighbor in $C$, the $KS$-partition $K=B$, $S=C$ is an $S$-max partition of $G[B\cup C]$. 

Let $H \in ({\cal T})_{\leq n-2}$.   By Theorem~\ref{split-thm}, we may choose a $KS$-partition of $H$ that is $S$-max. Hence every $b\in K$ has a neighbor in $S$ (otherwise $b$ could be added to $S$).  Define $g:({\cal T})_{\leq n-2}\rightarrow $ \NGonemtwo$_n$ by $g(H)=G $ where $G$ is the graph formed from $H$ by 
adding a clique $A$ of $n - |V(H)| $ new vertices and joining every vertex in $A$ 
  to every vertex in $K$.  Note that $G$ has $n$ vertices, and $G$ is an split 
graph with $KS$-partition $K'=A\cup K$ and $S$.   For $a \in A$, the sets $K' - \{a\}$ and $S \cup \{a\}$ provide another $KS$-partition of $G$, so $G$ is an unbalanced split graph by Proposition~\ref{unique-bal-prop}.
  
%Therefore $G$ is an unbalanced split 
%graph because any $a\in A$ is not adjacent to all vertices in $S$, so $a$ can be added to $S$ to make a larger stable set. A split graph with distinct $K$-max and $S$-max $KS$-partitions is unbalanced by Remark~\ref{Un-rem}.
Thus $G\in $ \NGonecuptwo$_n$ by Theorem~\ref{NG-unbalanced}. Since $\chi(G)=\omega(G)= |A| + |K|$, the ABC-partition of $G$ is $A\cup K \cup S$ and $G\in$  \NGonemtwo$_n$ because $A$ is a clique and $|A|\geq 2$. It is straightforward to check that $g$ is the inverse function of $f$.

\medskip

%Proof of (2):  Complementation provides a bijection between \NGonemtwo$_n$ and 
%\NGtwomone$_n$. 

%$(\NGtwo)$_n$ with $|A| \ge 2$ is the complement of an \ngone graph with $|A| \ge 2$ (and vice versa).    
%Furthermore, the class $({\cal T})_{\leq n-2}$ is closed under taking graph complements.  Thus we can use the bjijection in (1) together with taking graph complements to prove (2).

%Since we have shown a bijection between $\{ G$ is an \ngone graph with $|A|\geq 2 \}$ and $({\cal T})_{\leq n-2}$, and complementation provides a bijection between  $\{ G $ is an \ngtwo graph with $|A|\geq 2 \}$ and $\{ G$ is an \ngone graph with $|A|\geq 2 \}$, then (2) is proved.

\noindent {\bf Proof of (2):}  Let  $G\in $ (\NGthree)$_n$.   Define $f:$ (\NGthree)$_n\rightarrow ({\cal T})_{\leq n-5}$ by $f(G) = G[(B-B')\cup C]$.  We will show $f$ is our desired bijection.  Since $|A| =5$, the graph $f(G)$ has at most $n-5$ vertices, and it is a split graph with $KS$-partition $K=B-B'$, $S=C$.  Thus $f(G) \in ({\cal T})_{\leq n-5}$.   Furthermore, since every $b\in B-B'$ has a neighbor in $C$, the $KS$-partition $K=B-B'$, $S=C$ is an $S$-max partition of $f(G)$. 

Let $H \in ({\cal T})_{\leq n-5}$.   By Theorem~\ref{split-thm}, we may choose a $KS$-partition of $H$ that is $S$-max. Hence every $b\in K$ has a neighbor in $S$.  Define $g:({\cal T})_{\leq n-5}\rightarrow $ (\NGthree)$_n$ by $g(H)=G $ where $G$ is the graph formed from $H$ by 
adding a set $D$ of $n - |V(H)| $ new vertices such that $G[D]$ is a clique minus the edges of a 5-cycle, and  every vertex in $D$ is
adjacent to every vertex in $K$ but to no vertices in $S$. Note that $G$ has $n$ vertices, and $\chi(G)=|D|-2+|K|$ since the vertices in $D$ in the 5-cycle may be colored with 3 instead of 5 colors. Five 
vertices in $D$ have degree $|D|-3+|K|$ and the rest have degree $|D|-1+|K|$. For all 
$b\in K$, $deg_G(b)> |D|+|K|-1$ and for all $c\in S$, $deg(c)\leq |K|$. Thus  in the ABC-partition of $G$, $A$ is the set of five vertices in $D$ that induce a 5-cycle,    $B=(D-A)\cup K$ and $C=S$.  It follows  from Definition~\ref{NG-123-def} that $G\in $ 
(\NGthree)$_n$  and it is straightforward to check that $g$ is the inverse function of $f$.
 \end{proof}

In the proof of Theorem~\ref{glor-thm} we used a more complicated bijection in part (2) because the function $f(G) = G[B\cup C]$ is not injective when applied to graphs in (\NGthree)$_n$.  This is illustrated in Example~\ref{non-inj-ex}.

\begin{example}
\label{non-inj-ex}
{\rm
Let $G $ be the $NG$-3 graph where $A_G= \{a_1,a_2,a_3,a_4,a_5\}$,  $B_G = \{b_1,b_2,b_3\}$, $C_G = \{c_1,c_2\}$ and   $b_1c_1$ is the only edge between $B_G$ and $C_G$.  Note that $\chi(G) = 6$, the vertices in $A_G$ have degree 5, those in $B_G$ have degree greater than 5, and those in $C_G$ have degree less than 5, as needed.  The graph $G - A_G$ consists of a  triangle (with vertices $b_1,b_2,b_3$), an edge between $b_1$ and $c_1$ and an isolated vertex ($c_2$).    

Now let $H $ be the $NG$-3 graph where again $A_H = \{a_1,a_2,a_3,a_4,a_5\}$, $B_H = \{b_1,b_2\}$, $C_H = \{c_1,c_2, c_3\}$ and the edge set between $B_H$ and $C_H$ is $\{b_1c_1, b_1c_3, b_2c_3\}$.     Now   $\chi(G) = 5$, the vertices in $A_H$ have degree 4, those in $B_H$ have degree greater than 4, and those in $C_H$ have degree less than 4, as needed.  
The graph $H - A_H$ consists of a  triangle (with vertices $b_1,b_2,c_3$), an edge between $b_1$ and $c_1$ and an isolated vertex ($c_2$).      Thus $G$ and $H$ are \emph{not} isomorphic, yet $G - A_G$ and $H - A_H$ \emph{are} isomorphic.
}
\end{example}

Theorem~\ref{table-thm} includes a bijection between the classes $({\cal U})_{n-4}$ and $({\cal T})_{\leq n-5}$, which implies
   that $|({\cal U})_n|=|({\cal T})_{\leq n-1}|$.
We provide a second proof of this equality using the classic presentation of split graphs found in \cite{Go80} and this second proof does not rely on   NG-graphs.  
   Let $({\cal U})_n^K$ be the set of triples $(G,K_G,S_G)$ where $G$ is an unlabeled, unbalanced split graph on $n$ vertices, and $K_GS_G$ is a KS-partition that is K-max.  Similarly, define
  $({\cal U})_n^S$ where the KS-partition  $K_GS_G$ is   S-max.  Recall from Proposition~\ref{unique-bal-prop}  that balanced split graphs have a unique KS-partition, so we may define $({\cal B})_n^{KS}$ to be the set of triples $(G,K_G,S_G)$ of balanced split graphs on $n$ vertices together with their unique KS-partition.  

\begin{Thm} There is a bijection between $({\cal U})_n^K$ and $({\cal U})_{n-1}^K\cup ({\cal U})_{n-1}^S\cup {\cal B}^{KS}_{n-1}$, and $|({\cal U})_n| = |({\cal T})_{\leq n-1}|$.
\label{fixed-part}
\end{Thm}
\begin{proof} Let $(G,K_G,S_G)\in ({\cal U})_{n-1}^K\cup ({\cal U})_{n-1}^S\cup B^{KS}_{n-1}$.  Create a new graph $H$ consisting of $G$ together with a new vertex $w$ that is adjacent to every vertex in $K_G$.     Then $H$ is a split graph and the sets    $K=K_G \cup \{w\} $ and  $S=S_G$ form a KS-partition of $H$.  
 Further, $H$ is an unbalanced split graph because $w$ is not adjacent to any vertex in $S$, hence $K' = K_G  $, $S' = S_G \cup \{w\}$ is another KS-partition of $H$.   Thus, $KS$ is a K-max partition of $H$.   Define $\phi: ({\cal U})_{n-1}^K\cup ({\cal U})_{n-1}^S\cup B^{KS}_{n-1}\rightarrow ({\cal U})_n^K$ by $\phi(G) = H$.

We next show that $\phi$ is a reversible map.   Take any $(H, K_H, S_H) \in  ({\cal U})_n^K$.    By Theorem~\ref{split-thm}, there exists $w\in K_H$ so that $K_H - \{w\}$ is complete and $S_H \cup \{w\}$ is a stable set.    Let $\psi(H) = H-w$.    Note that $\psi(H)$ is the same unlabeled graph  for any possible choice of $w$.
 Then $H-w$ is a split graph on $n-1$ vertices and $K =   K_H-\{w\}$, $S = \ S_H$ is a KS-partition of $H-w$.  
   Therefore, 
$\psi$ reverses the operation of $\phi$.

Finally, we prove $|({\cal U})_n| = |({\cal T})_{\leq n-1}|$.    The sets $({\cal U})_{n-1}^k$,  $({\cal U})_{n-1}^S$, and $ B^{KS}_{n-1}$ are disjoint by definition, so $|({\cal U})_n^K|= |({\cal U})_{n-1}^k| +  |({\cal U})_{n-1}^S| + | B^{KS}_{n-1}|.$  We know from Corollary~\ref{two-unbal-prop} that every unbalanced split graph has exactly two non-isomorphic KS-partitions, one K-max and the other S-max, thus $|({\cal U})_n| = |({\cal U})_n^K| = |({\cal U})_n^S|$.  Balanced split graphs have a unique KS-partition (Proposition~\ref{unique-bal-prop}) thus $|({\cal B})_n| = |({\cal B})_n^{KS}| $.  Thus

  \[|({\cal U})_n|= |({\cal U})_n^K| =  |{\cal U}^K_{n-1}|+ |{\cal U}^S_{n-1}|+ |{\cal B}^{KS}_{n-1}|= 2 |({\cal U})_{n-1}|+|({\cal B})_{n-1}|.\]
   \[|({\cal U})_n| =
  |({\cal U})_{n-1}|+|({\cal S})_{n-1}|.\] Since $|({\cal T})_{\leq n-1}|=|({\cal T})_{\leq n-2}|+|({\cal S})_{n-1}|$, and $|{\cal U}_1|=1=|{\cal T}_0|$, we see that $|({\cal U})_n|$ and $|({\cal T})_{\leq n-1}| $ satisfy the same recurrence and initial condition and therefore are equal.
\end{proof}

\section{Counting}
\label{counting}

In this section we use the bijections from Section~\ref{bijections-sec} to calculate the size of classes of split graphs,  NG-graphs and pseudo-split graphs. 
In \cite{Cl90}, Clarke gives an expression for the number of minimal $k$-covers of a set of $n$ indistinguishable objects.  Royle \cite{Ro00}  then describes a bijection between such $k$-covers and the set of split graphs on $n$ vertices with exactly $k$ maximal cliques that each contain a vertex in none of the other maximal cliques.  Summing over $k$ from 1 to $n$, Royle obtains a formula for $|({\cal S})_n|$, the number of split graphs on $n$ vertices.  Unfortunately, the formula is quite complicated.  %Royle used Maple to construct a table of values for $|({\cal S})_n|$ for $n$ between 1 and 20.  

Our bijections from Section~\ref{bijections-sec} allow us to calculate the number of balanced and unbalanced split graphs,  all categories of NG-graphs and pseudo-split graphs  solely in terms of the number of split graphs.  Let $(\mathcal{PS})_n$ denote the set of pseudo-split graphs on $n$ vertices.  The corollary below describes the exact formulas.

\begin{Cor}
\label{cor-formulas}
 The following equalities hold for each $n \ge 1$.
\begin{enumerate}
%\item  $|({\cal T})_{n}| = \sum_{i=0}^n |(\mathcal{S}_i)|$ 
\item[(1)]  $|(\mathcal{U})_n| = \sum_{i=0}^{n-1} |(\mathcal{S}_i)|$ 
\item[(2)]  $|(\mathcal{B})_n| =  |(\mathcal{S})_n| -   \sum_{i=0}^{n-1} |(\mathcal{S}_i)|$ 
\item[(3)]  $|$(\NGone)$|_n$ $=$  $|$(\NGtwo)$|_n$  $=  |(\mathcal{S})_{n-1}|$
\item[(4)]   $|$(\NGthree)$|_n$ $=   \sum_{i=0}^{n-5} |(\mathcal{S}_i)|$ 
\item[(5)]  $|(\mathcal{NG})_n| =   \sum_{i=0}^{n-1} |(\mathcal{S}_i)| + \sum_{i=0}^{n-5} |(\mathcal{S}_i)|$ 
\item[(6)]   $|(\mathcal{PS})_n| =  |(\mathcal{S})_n| + \sum_{i=0}^{n-5} |(\mathcal{S}_i)|$ 
\end{enumerate}
\end{Cor}  

\begin{proof}
The formulas in (1), (3) and (4) follow directly from the bijections in Section~\ref{bijections-sec}.  To prove (2), we use (1) and the fact that a split graph is either balanced or unbalanced, and to prove    (5)  we apply (1) and (4) and Remark~\ref{Un-rem}.  Finally, to prove (6), we use Remark~\ref{pseudo-rem} and apply the formula from (4).  
%We know $|({\NGcal})_n| =  |$\NGonecuptwo$_n|  + |$(\NGthree)$_n|$ and  $|$(\NGthree)$_n| = |({\cal T})_{\leq n-5}|$ by Theorem~\ref{glor-thm}.  The remaining equality, $|$\NGonecuptwo$_n| = |({\cal T})_{\leq n-1}|$, follows from  
% Remark~\ref{Un-rem}  and  the bijections in row (c) of Table~\ref{bij-table}.
\end{proof}

%$|({\cal S})_n|$.  

Recall that $|({\cal T})_{\le n}| = \sum_{i=0}^n |(\mathcal{S}_i)|$.  In Table~\ref{counting-table}, we make use of the values computed by Royle \cite{Ro00} for $|({\cal S})_n|$  and the formulas in Corollary~\ref{cor-formulas} to determine the values for $|({\cal T})_{\le n}|$, $|(\mathcal{U})_n|$, $|(\mathcal{B})_n|$, $|(\mathcal{NG})_n|$ and   $|(\mathcal{PS})_n|$ for $n = 0, \hdots, 11$.  The table  in \cite{Ro00} shows  rapid growth in  the number of split graphs.  Royle  does not divide split graphs into balanced and unbalanced as we do.  Table~\ref{ratio-table} shows the ratio of  the number of balanced split graphs to the number of  split graphs on $n$ vertices, indicating that the rapid growth  in the number of split graphs comes from the balanced category.    We conjecture  
that  this ratio approaches 1 as $n$ goes to infinity.  
 We now show why this may  be the case.
 
 %We would like to thank Ke Chen \cite{Chen} and David Constantine \cite{Constantine} for assistance with the following proof. 

\begin{Thm}
\label{convergence}
If $\lim_{n \rightarrow \infty}  \frac{|(\mathcal{S})_{n-1}|}{|(\mathcal{S})_{n}|}  \rightarrow  0$  then  $\lim_{n \rightarrow \infty} \frac{|(\mathcal{B})_{n}|}{|(\mathcal{S})_{n}|} \rightarrow  1$.
\end{Thm}

%\begin{proof}   
%First, it is easy to see that $|(\mathcal{S})_n|$ is non-decreasing as $n$ increases.   Thus,  
%$$  \frac{|(\mathcal{S})_{i}|}{|(\mathcal{S})_{n}|}  \leq  \frac{|(\mathcal{S})_{n-1}|}{|(\mathcal{S})_{n}|} \mbox{\hspace*{.1in} for $i = 0, \hdots, n-1$.}$$

%\noindent  It follows that if  $ \frac{|(\mathcal{S})_{n-1}|}{|(\mathcal{S})_{n}|} \rightarrow 0$ as $n \rightarrow \infty$ then   $  \frac{|(\mathcal{S})_{i}|}{|(\mathcal{S})_{n}|} \rightarrow 0$ as $n \rightarrow \infty$ for $i = 0, \hdots, n-1$.  
%First, we note that $|(\mathcal{B})_{n}|/{|(\mathcal{S})_{n}|} \leq 1$ for any $n$.   

%Applying formula (2)  from Corollary~\ref{cor-formulas},  $|(\mathcal{B})_n| =  |(\mathcal{S})_n| -   \sum_{i=0}^{n-1} |(\mathcal{S}_i)|$,  it follows that 
%\begin{eqnarray*}
%  \frac{|(\mathcal{B})_{n}|}{|(\mathcal{S})_{n}|} & = & 1 - \sum_{i=0}^{n-1} \frac{ |(\mathcal{S}_i)|}{|(\mathcal{S}_n)|}. 
%     & \ge &  1 - (n-1)  \frac{|(\mathcal{S})_{n-1}|}{|(\mathcal{S})_{n}|} \\
%\end{eqnarray*}
%Since we assume that $\lim_{n \rightarrow \infty}  \frac{|(\mathcal{S})_{n-1}|}{|(\mathcal{S})_{n}|}  \rightarrow  0$,  each term in the summation converges to $0$ as well so  $\lim_{n \rightarrow \infty} \frac{|(\mathcal{B})_{n}|}{|(\mathcal{S})_{n}|}  \rightarrow 1$.  

\begin{proof} From Table~\ref{bij-table}, $|(\mathcal{U})_{n}|=|(\mathcal{T})_{\leq n-1}|$, hence $|(\mathcal{U})_{n}|-|(\mathcal{U})_{n-1}|=|(\mathcal{S})_{n-1}|$. Thus,
\[|(\mathcal{B})_{n}|-|(\mathcal{B})_{n-1}|= (|(\mathcal{S})_{n}|-|(\mathcal{U})_{n}|) -( |(\mathcal{S})_{n-1}|-|(\mathcal{U})_{n-1}|)\]\[= (|(\mathcal{S})_{n}|-|(\mathcal{S})_{n-1}|) -(|(\mathcal{U})_{n}| -|(\mathcal{U})_{n-1}|)=|(\mathcal{S})_{n}|-2|(\mathcal{S})_{n-1}|.\]
Thus,   by the hypothesis and the Sandwich Theorem,

\[1\geq \lim_{n\rightarrow\infty} \frac{|(\mathcal{B})_{n}|}{|(\mathcal{S})_{n}|} \geq \lim_{n \rightarrow \infty}   \frac{|(\mathcal{B})_{n}| - |(\mathcal{B})_{n-1}|   }{|(\mathcal{S})_{n}| }   =   \lim_{n \rightarrow \infty}   \frac{ |(\mathcal{S})_n| -  2  |(\mathcal{S})_{n-1}|}{|(\mathcal{S})_{n}| }  
\]\[ =   \lim_{n \rightarrow \infty}   1  -  2\left(\frac{|(\mathcal{S})_{n-1}|}{|(\mathcal{S})_{n}|} \right) 
  =    1.\]
%\end{eqnarray*}

\end{proof}

%If as $n \rightarrow \infty$, the fraction ${|(\mathcal{S})_{n-1}|}/{|(\mathcal{S})_{n}|}$ converges to $0$  then it follows that $|(\mathcal{B})_{n}|/{|(\mathcal{S})_{n}|}$  converges to $1$ in the limit.   \end{proof}

%We conjecture that this ratio approaches 1 as $n$ goes to infinity.  

%use this formula for $|({\cal S})_n|$ to calculate the number of balanced and unbalanced split graphs and all categories of NG-graphs discussed in this paper.   These quantities are shown in Table~\ref{counting-table}.

\begin{table}
\begin{tabular}{|c||r|r|r|r|r|r|r|r|r|r|r|r|}\hline  
$n$ &0  &1 &2 &3 &4 &5  &6   &7   \ &8\ \  &9\ \ \  &10\ \ \ \  & 11\ \ \ \ \\ \hline  \hline
$|({\cal S})_n|$ &1 &1 &2 & 4&9 &21 &56 &164 &557 &2,223 &10,766 & 64,956 \\  \hline 
$|({\cal T})_{\le n}|$&1 &2 &4 &8 &17 &38 &94 &258 &815 &3,038 &13,804 &78,760\\  \hline 
$|({\cal U})_n|$ &0 &1 &2 &4 &8 &17 &38 &94 &258 &815 &3,038 &13,804 \\  \hline 
$|({\cal B})_n|$ & 1 &0 &0 &0 &1 &4 &18 &70 &299 &1,408 &7,728 & 51,152 \\  \hline 
$|{(\NGcal)}_n|$&0 &1 &2 &4 &8 &18 &40 &98 &266 &832 &3,076 & 13,898\\  \hline 
$|(\mathcal{PS})_n|$ & 1 & 1 &2 & 4&9 & 22 & 58 & 168 & 565 & 2,240 & 10,804 & 65,050 \\ \hline
\end{tabular}
\caption{The number of split graphs (total, balanced, unbalanced) and NG-graphs on $n$ vertices.}
\label{counting-table}
\end{table}

  \begin{table}
\begin{tabular}{|c||r|r|r|r|r|r|r|r|r|r|r|r|r|}\hline  
$n$ &4  &5 &6 &7 &8 &9 &10 &11 &12 &13 &14 & 15& 16 \\ \hline  \hline
ratio & .11 & .19 &.32 & .42& .54& .63& .72&.79 & .84&.89 & .92& .94 &.96\\  \hline 
%$|({\cal B})_n|/|({\cal S})_n|$ & .11 & .19 &.32 & .42& .54& .63& .72&.79 & .84&.89 & .92& .94 \\  \hline 
\end{tabular}
\caption{The ratio of the number of  balanced split graphs  to split graphs on $n$ vertices.}
\label{ratio-table}
\end{table}

\section{Acknowledgments}

We would like to thank Ke Chen \cite{Chen} and David Constantine \cite{Constantine} for their assistance in proving Theorem \ref{convergence}.

\end{document}